\documentclass{amsart}
\title{Donaldson-Thomas Theory and Resolutions of Toric $A$-Singularities}
\author{Dustin Ross}
\address{Dustin Ross, University of Michigan, Department of Mathematics, Ann Arbor, MI 48109, USA}
\email{dustyr@umich.edu}

\usepackage {color,graphicx,psfrag, verbatim,amssymb,amscd,enumerate,subfigure,ytableau,tikz}
\usepackage{texdraw,xypic,appendix,xcolor}
\definecolor{lgray}{gray}{0.8}
\definecolor{dgray}{gray}{0.4}
\definecolor{Gold}{rgb}{1.,0.84,0.}
\DeclareRobustCommand{\rchi}{{\mathpalette\irchi\relax}}
\newcommand{\irchi}[2]{\raisebox{\depth}{$#1\chi$}}
\renewcommand{\top}{\mathrm{top}}

\newcommand{\bP}{\mathbb{P}}

\newcommand{\bZ}{\mathbb{Z}}

\newcommand{\cO}{\mathcal{O}}

\newcommand{\bC}{\mathbb{C}}

\newcommand{\fq}{\mathfrak{q}}
\newcommand{\bq}{\mathbf{q}}
\newcommand{\cX}{\mathcal{X}}
\newcommand{\cN}{\mathcal{N}}

\newcommand{\cZ}{\mathcal{Z}}
\newcommand{\bv}{\mathbf{v}}
\newcommand{\cRP}{\mathcal{RP}}
\DeclareRobustCommand*\circled[1]{\tikz[baseline=(char.base)]{
            \node[shape=circle,draw,inner sep=1.5pt] (char) {#1};}}

\newcommand\xaxis{210}
\newcommand\zaxis{-30}
\newcommand\yaxis{90}

\newcommand\topside[4]{
  \fill[fill=#4, draw=black,shift={(\xaxis:#1)},shift={(\yaxis:#2)},
  shift={(\zaxis:#3)}] (0,0) -- (30:1) -- (0,1) --(150:1)--(0,0);
}

\newcommand\leftside[4]{
  \fill[fill=#4, draw=black,shift={(\xaxis:#1)},shift={(\yaxis:#2)},
  shift={(\zaxis:#3)}] (0,0) -- (0,-1) -- (210:1) --(150:1)--(0,0);
}

\newcommand\rightside[4]{
  \fill[fill=#4, draw=black,shift={(\xaxis:#1)},shift={(\yaxis:#2)},
  shift={(\zaxis:#3)}] (0,0) -- (30:1) -- (-30:1) --(0,-1)--(0,0);
}

\newcommand\cube[4]{
  \topside{#1}{#2}{#3}{#4} \leftside{#1}{#2}{#3}{#4} \rightside{#1}{#2}{#3}{#4}
}

\newtheorem{dummy}{}[section]
\newtheorem{lemma}[dummy]{Lemma}
\newtheorem{proposition}[dummy]{Proposition}
\newtheorem{theorem}[dummy]{Theorem}

\newtheorem{conjecture}{Conjecture}

\newtheorem*{observation}{Observation}
\newtheorem*{maintheorem}{Main Theorem}

\theoremstyle{definition}

\newtheorem{definition}[dummy]{Definition}
\newtheorem{example}[dummy]{Example}

\newtheorem{remark}[dummy]{Remark}

\usepackage{graphicx}

2
\begin{document}

%\subjclass[2010]{Primary 14n35. Secondary  05a15.}

\begin{abstract}
We prove the crepant resolution conjecture for Donaldson-Thomas invariants of toric Calabi-Yau $3$-orbifolds with transverse $A$-singularities.
\end{abstract}

\maketitle
%\tableofcontents

\section{Introduction}

\subsection{Summary of results}

Motivated by ideas in mirror symmetry, Ruan's crepant resolution conjecture roughly states that the Gromov-Witten invariants of a Calabi-Yau orbifold $\cZ$ should be related to those of a crepant resolution $\pi:W\rightarrow\cZ$. Due to the conjectural relationship between Gromov-Witten theory and Donaldson-Thomas theory in dimension $3$, first formulated by Maulik--Nekrasov--Okounkov--Pandharipande \cite{mnop:gwdti}, one might expect that a similar relationship holds for Donaldson-Thomas invariants of $3$-folds. In particular, the Gromov-Witten crepant resolution conjecture has an especially nice formulation when $\cZ$ satisfies the hard-Lefschetz condition, due to Bryan--Graber \cite{bg:crc}, and this led Bryan--Cadman--Young to conjecture a similar formulation of the Donaldson-Thomas crepant resolution conjecture for hard-Lefschetz $3$-orbifolds \cite{bcy:otv}. Explicitly, they made the following conjecture.

\begin{conjecture}[\cite{bcy:otv} Conjecture 1]\label{conj:crc}
If $\pi:W\rightarrow\cZ$ is a crepant resolution of a hard-Lefschetz $3$-orbifold, then there is an explicit change of variables such that
\[
DT_{mr}(\cZ)=\frac{DT(W)}{DT_{exc}(W)}
\]
where $DT_{mr}(-)$ denotes the reduced, multi-regular Donadson-Thomas potential and $DT_{exc}(W)$ is obtained by restricting the Donaldson-Thomas potential of $W$ to curves supported on the exceptional locus of $\pi$ (see Section \ref{sec:dt} for precise definitions).
\end{conjecture}

The simplest type of hard-Lefschetz orbifold in dimension $3$ occurs when the orbifold structure is cyclic and supported on isolated curves in $\cZ$, we call these \emph{transverse $A$-singularities}. The main result of this paper is the following. 

\begin{maintheorem}[Theorem \ref{thm:globalcrc}]
Conjecture \ref{conj:crc} is true for toric Calabi-Yau $3$-orbifolds with transverse $A$-singularities.
\end{maintheorem}

We refer the reader to Section \ref{sec:toricdt} for an explicit description of the change of variables. We mention that a comparison formula very similar in nature to Conjecture \ref{conj:crc} has previously been proved by Calabrese \cite{c:crcdti}, for general hard-Lefschetz $3$-orbifolds. However, at this time there is still a gap in recovering Conjecture \ref{conj:crc} explicitly from Calabrese's formula. Yet another approach to this problem has been studied by Bryan--Steinberg \cite{bs:ccifcc}.
 
The objects of interest in this paper are inherently algebro-geometric. However, the methods we employ are combinatorial in nature. More specifically, Bryan--Cadman--Young \cite{bcy:otv}, generalizing Okounkov--Reshetikhin--Vafa \cite{orv:qcycc} and Maulik--Nekrasov--Okounkov--Pandharipande \cite{mnop:gwdti}, showed that the entire Donaldson-Thomas theory of toric Calabi-Yau $3$-orbifolds can be recovered from a basic building block, the \emph{orbifold topological vertex}. The orbifold topological vertex is a generating function of colored 3D partitions associated to each torus fixed point in $\cZ$  and the Donaldson-Thomas potential can be recovered from the orbifold vertex via an explicit \emph{gluing algorithm}. %\textcolor{red}{In the transverse $A$-singularity case, the orbifold vertex has an explicit expression in terms of loop Schur functions,  reviewed in Section \ref{sec:toricdt}.}

In order to prove Theorem \ref{thm:globalcrc}, we proceed in two steps. We first formulate and prove a local correspondence on the level of the orbifold topological vertex (Theorem \ref{thm:vertexcrc}). Using the vertex operator expression for the orbifold topological vertex, developed by Bryan-Cadman-Young \cite{bcy:otv}, the local correspondence is proved by commuting vertex operators and careful book-keeping of the resulting commutation relations. The second step, carried out in Section \ref{sec:globalcrc}, is to prove that the vertex CRC is compatible with the edge terms in the gluing algorithm.

\subsection{Relation to other work}

There is a conjectural diagram of equivalences for hard-Lefschetz $3$-orbifolds:
\[
\begin{CD}
GW(W) @= DT(W)\\
@| @|\\
GW(\cZ) @= DT(\cZ)
\end{CD}
\]
where the horizontal equivalences are the Gromov-Witten/Donaldson-Thomas correspondence and the vertical equivalences are the crepant resolution conjecture, all of which conjecturally consist of a formal change of variables in the generating series along with analytic continuation of the parameters. In the case of toric targets with transverse $A$-singularities, the top equivalence is a theorem of Maulik--Oblomkov--Okounkov--Pandharipande \cite{moop:gwdtc} (more generally they proved it for all toric $3$-folds) and the equality on the right is the main theorem of this paper. The bottom conjectural equivalence was made explicit in \cite{rz:ggmv,rz:chialsf} and proved in the case where the toric orbifold is a local surface. 

In \cite{bcr:craos}, Brini--Cavalieri--Ross proved an all-genus crepant resolution conjecture for the open Gromov-Witten theory of the $A_{n-1}$ vertex, which provides a local version of the equality on the left. In \cite{r:gwdtcrc}, we prove that the main result of \cite{bcr:craos} along with the results of this paper and the correspondence of gluing algorithms of \cite{rz:chialsf} are sufficient to deduce the bottom equality and, hence, allow us to ``complete the square'' for $\cZ$ a toric Calabi-Yau $3$-orbifold with transverse $A$-singularities. In particular, the bottom equivalence provides strong structural results about $GW(\cZ)$ that are not obvious from a purely Gromov-Witten perspective.

\subsection{Plan of the paper}

In Section \ref{sec:dt}, we review the basic definitions of Donaldson-Thomas theory in order to make precise the objects which play a role in Conjecture \ref{conj:crc}. We pay particularly close attention to the case of toric targets with transverse $A$-singularities in Section \ref{sec:toricdt}, where we develop an explicit statement of the crepant resolution conjecture. In Section \ref{sec:orbvertex}, we review the orbifold topological vertex of Bryan--Cadman--Young and we recall the gluing algorithm in Section \ref{sec:gluing}. In Section \ref{sec:vertexcrc}, we state the local correspondence for the orbifold topological vertex and we prove is by a careful manipulation of vertex operators. In Section \ref{sec:globalcrc}, we show that the local correspondence is compatible with the gluing algorithm and deduce the Donaldson-Thomas crepant resolution conjecture for all toric Calabi-Yau $3$-orbifolds with transverse $A$-singularities.

\subsection{Acknowledgements} The author is greatly indebted to Jim Bryan and Ben Young for helpful conversation and encouragement. He is also grateful to Renzo Cavalieri for carefully listening to the main arguments appearing in this paper and providing helpful feedback. The author has been supported by NSF RTG grants DMS-0943832 and DMS-1045119 and the NSF postdoctoral research fellowship DMS-1401873.

\section{Donaldson-Thomas theory}\label{sec:dt}

In this section, we review the basic definitions of Donaldson-Thomas (DT) theory for Calabi-Yau (CY) $3$-orbifolds.

\subsection{General theory}\label{sec:definitions}

Let $\cZ$ be a CY $3$-orbifold, ie. a smooth, quasi-projective Deligne-Mumford stack of dimension three over $\bC$ with generically trivial isotropy and trivial canonical bundle. Let $Z$ denote the coarse moduli space of $\cZ$.

Let $F_1K(\cZ)$ denote the compactly supported elements of $K$-theory, up to numerical equivalence, supported in dimension at most one. Then for any $\gamma\in F_1K(\cZ)$, the corresponding DT invariant is defined as a weighted Euler characteristic
\begin{equation}\label{eqn:dtdef}
DT_\gamma(\cZ):=\sum_{k\in\bZ}k\rchi_{\top}(\nu^{-1}(k))
\end{equation}
where $\nu:\text{Hilb}_\gamma(\cZ)\rightarrow\bZ$ is Behrend's constructible function \cite{b:dttivmg} associated to the Hilbert scheme of substacks $V\subset\cZ$ with $[\cO_V]=\gamma$, and $\rchi_{\top}(-)$ is the topological Euler characteristic. Define the \emph{DT potential} as the formal series
\[
\widehat{DT}(\cZ):=\sum_{\gamma\in F_1K(\cZ)}DT_\gamma(\cZ)q^\gamma
\]
with formal parameter $q$ defined so that $q^{\gamma_1}q^{\gamma_2}:=q^{\gamma_1+\gamma_1}$.

For the purposes of Conjecture \ref{conj:crc}, we work with a certain specialization of $\widehat{DT}(\cZ)$. In order to define this specialization, we consider two important subgroups of $F_1K(\cZ)$. The \emph{multi-regular $K$-group} $F_{mr}K(\cZ)\subseteq F_1K(\cZ)$ is defined to consist of classes represented by sheaves such that at a general point of each curve in the support, the associated representation of the stabilizer is a multiple of the regular representation. The \emph{zero-dimensional $K$-group} $F_0K(\cZ)\subseteq F_1K(\cZ)$ is defined to consists of classes represented by sheaves with zero-dimensional supports. We have specialized series
\[
\widehat{DT}_{mr}(\cZ):=\sum_{\gamma\in F_{mr}K(\cZ)}DT_\gamma(\cZ)q^\gamma \hspace{.5cm}\text{and}\hspace{.5cm} DT_{0}(\cZ):=\sum_{\gamma\in F_{0}K(\cZ)}DT_\gamma(\cZ)q^\gamma
\]
and we define the \emph{reduced, multi-regular DT potential} by
\[
DT_{mr}(\cZ):=\frac{\widehat{DT}_{mr}(\cZ)}{DT_0(\cZ)}.
\]

Assume $\cZ$ satisfies the \emph{hard-Lefschetz condition} (cf. \cite{bg:crc}) and let $\pi:W\rightarrow Z$ be a crepant resolution by a smooth variety $W$.  Define $F_{exc}K(W)\subseteq F_1K(W)$ to consist of classes represented by sheaves supported on curves in the exceptional locus of $\pi$. We define
\[
DT_{exc}(W):=\sum_{\gamma\in F_{exc}K(W)}DT_\gamma(W)q^\gamma.
\] 

These definitions make precise the objects in the statement of Conjecture \ref{conj:crc}. Henceforth, we focus on the particular class of orbifolds of interest to us.

\subsection{Toric targets with transverse $A$-singularities}\label{sec:toricdt}

In this section we describe some of the basic geometry of toric CY $3$-orbifolds with transverse $A$-singularities. This description allows us to choose a basis for the relevant $K$-groups so we can explicitly describe the change of variables in the crepant resolution conjecture.

\subsubsection{Global geometry}

Let $\cZ$ be a toric CY $3$-orbifold with transverse $A$-singularities (ie. $\cZ$ has cyclic isotropy supported on disjoint torus-fixed lines), and let $W$ be its toric resolution (described more explicitly below).  Then to $\cZ$ (and $W$) we can associate a \emph{web diagram}, a trivalent planar graph 
\[
\Gamma_\cZ=\{\text{Edges: } E_\cZ, \text{Vertices: } V_\cZ\}
\]
where vertices correspond to torus fixed points in $\cZ$, edges correspond to torus invariant lines, and regions delineated by edges correspond to torus invariant divisors. The web diagram is essentually dual to the toric fan of $\cZ$. Additionally, we choose an orientation for each edge of $\Gamma_\cZ$.  Let $n_e$ denote the order of the isotropy on the line $L_e$ corresponding to an edge $e$.  We label the edges adjacent to each vertex $(e_1(v),e_2(v),e_3(v))$ requiring that
\begin{itemize}
\item if $v$ is adjacent to an edge $e$ with $n_e>1$, then $e_3(v)=e$, and
\item the labels $(e_1(v),e_2(v),e_3(v))$ are ordered counterclockwise.
\end{itemize}

In order to formulate the change of variables in the crepant resolution conjecture, we must define a few additional factors at each edge. The normal bundle splits $\cN_{L_e/\cZ}\cong\cN_r\oplus\cN_l$ where $\cN_r$ ($\cN_l$) corresponds to the normal bundle summand in the direction of the torus invariant divisor to the right (left) of $e$.  Let $p$ be a general point on $L_e$ and $p_0,p_\infty$ the torus-fixed points with corresponding vertices $v_0,v_\infty$ the initial and terminal vertices of $e$.  Because we are restricting to CY transverse $A$-singularities, a local neighborhood of the torus-fixed points can be expressed as a global quotient $\left[\bC^3/\bZ_n\right]$ where a generator $\xi_n$ of the cyclic group acts on the coordinates with weights $(1,-1,0)$. This allows us to write
\[
\cN_l=\cO(m[p]-\delta_0[p_0]-\delta_\infty[p_\infty])
\]
\[
\cN_r=\cO(m'[p]-\delta_0'[p_0]-\delta_\infty'[p_\infty])
\]
where $\delta_0=1$ if the edge corresponding to the fiber of $\cN_r$ over $p_0$ is labelled $e_3(v_0)$, and $\delta_0=0$ otherwise. Similarly, $\delta_0'=1$ if the edge corresponding to the fiber of $\cN_l$ over $p_0$ is labelled $e_3(v_0)$, and $\delta_0=0$ otherwise. Defining $\delta_\infty$ and $\delta_{\infty}'$ similarly, we have $\delta_\bullet+\delta_\bullet'\in\{0,1\}$ and the CY condition is equivalent to
\[
m+m'-(\delta_0+\delta_0'+\delta_\infty+\delta_\infty')=-2
\]

\begin{remark}
The $\delta$ factors defined here are different than those in \cite{bcy:otv} when $n=1$. We define them as such to eliminate the need for the factors in \cite{bcy:otv} involving $A_\lambda$.
\end{remark}

For each edge $e$, we define a set of formal variables
\[
\bq_e:=(q_{e,0},q_{e,1},\dots,q_{e,n_e-1})
\]
Geometrically, the $q_{e,i}$ index skyscraper sheaves supported on $e$ with $\bZ_{n_e}$ acting by the $i$th irreducible representation. Notice that the product
\[
q_{e,0}q_{e,1}\cdots q_{e,n_e-1}
\]
indexes the skyscraper sheaf with regular representation, which can be deformed away from $e$. Therefore, if we introduce a new variable $q$ indexing the skyscraper sheaf on a smooth point, we have the relations
\[
q:=q_{e,0}q_{e,1}\cdots q_{e,n_e-1},
\] 
for all $e$.

We also consider ``Novikov'' variables $v_e$ associated to each edge. The $v_e$ satisfy natural curve class relations coming from the geometry of $\cZ$. $DT(\cZ)$ is a formal series in the variables $\{v_e, q_{e,k}\}$ where $e\in E_\cZ$ runs over the set of edges in the web diagram and $0\leq k \leq n_e-1$. 

Now let $\pi:W\rightarrow \cZ$ be the toric resolution of $\cZ$. Then $\pi$ is an isomorphism away from the lines $L_e$ with $n_e>1$. For each edge $e\in E_\cZ$, there are $n_e$ corresponding edges in the web diagram for $W$ (see Figure \ref{fig:webdiagrams} and the corresponding caption). The orientation on $e$, which was used to compute $m_e$ above, induces an orientation on the $n_e$ corresponding edges in $E_W$ and we label them $f_{e,0},\dots,f_{e,n-1}$ from right to left. We label the edge connecting the initial point of $f_{e,k}$ to the initial point of $f_{e,k+1}$ by $g_{e,k+1}$ and we label the edge connecting the terminal point of $f_{e,k}$ to the terminal point of $f_{e,k+1}$ by $h_{e,k+1}$. Let $u_f$, $u_g$, $u_h$ be formal variables corresponding to these edges so that $DT(W)$ is a formal series in $\{u_f,u_g,u_h,q\}$. 

\begin{figure}
\begin{tikzpicture}
\begin{scope}[scale=1,xshift=6.5cm]
\draw (0,0) -- (1, .5) -- (2,1.5) -- (2,2.5) -- (1,3.5) -- (0,4);
\draw (1,.5) -- (5,.5);
\draw (2,1.5) -- (4,1.5);
\draw (2,2.5) -- (4,2.5);
\draw (1,3.5) -- (5,3.5);
\draw (6,0) -- (5,.5) -- (4,1.5) -- (4,2.5) -- (5,3.5) -- (6,4);
\end{scope}
\begin{scope}[scale=1, xshift=6.5cm]
\draw (3, .75)  node{$f_{e,0}$};
\draw (3,1.75) node{$f_{e,1}$};
\draw (3,2.75) node{$f_{e,2}$};
\draw (3,3.75) node{$f_{e,3}$};
\draw (1,1.1) node{$g_{e,1}$};
\draw (1.5,2) node{$g_{e,2}$};
\draw (1,2.9) node{$g_{e,3}$};
\draw (5,1.1) node{$h_{e,1}$};
\draw (4.5,2) node{$h_{e,2}$};
\draw (5,2.9) node{$h_{e,3}$};
\end{scope}
\begin{scope}[scale=1,xshift=.5cm]
\draw (1,1) -- (2,2) -- (1,3);
\draw (6,1) -- (5,2) -- (6,3);
\draw (3.5,2.25) node{$e$};
\end{scope}
\begin{scope}[scale=1, very thick ,xshift=.5cm]
\draw (2,2) -- (5,2);
\end{scope}
\end{tikzpicture}
\caption{The labeling of the web diagrams for orbifold (left) and resolution (right) near an edge $e$ with $n_e=4$ and all horizontal edges oriented rightward. The toric surfaces corresponding to the parallelograms are Hirzebruch surfaces $H_{4m_e+6}$,  $H_{4m_e+4}$, and $H_{4m_e+2}$ (from bottom to top) where $H_k:= \bP(\cO_{\bP^1}\oplus\cO_{\bP^1}(k))$. With this labeling, the bottom edge of each paralellogram corresponds to the zero section of the corresponding Hirzebruch surface.}\label{fig:webdiagrams}
\end{figure}
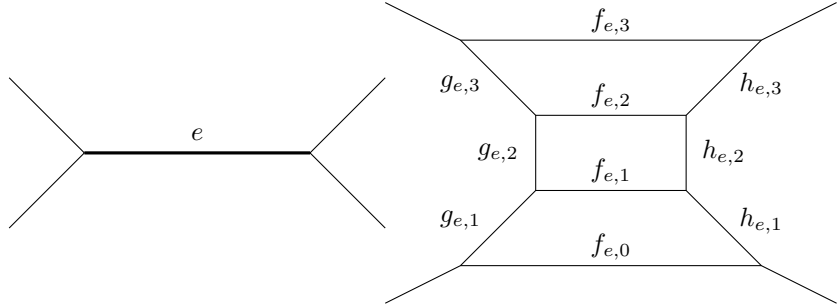

There are relations between the variables $u_f$, $u_g$, $u_h$ coming from the geometry of $W$. In particular, the edges $f_{e,k}$, $f_{e,k-1}$, $g_{e,k}$, and $h_{e,k}$ correspond to the toric boundary of the Hirzebruch surface 
\[
H_{n_em_e+2(n_e-k)}:=\bP(\cO_{\bP^1}\oplus\cO_{\bP^1}(n_em_e+2(n_e-k))),
\]
and thus satisfy the relations
\[
u_{g_{e,k}}=u_{h_{e,k}}
\]
and
\[
u_{f_{e,k}}=u_{f_{e,k-1}}u_{g_{e,k}}^{2k-2n_e-m_en_e}=u_{f_{e,0}}\prod_{l=1}^k u_{g_{e,l}}^{2l-2n_e-m_en_e}
\]
Henceforth, we interpret $u_{f_{e,k}}$ as a function of $u_{f_{e,0}}, u_{g_{e,1}},\dots u_{g_{e,k}}$. With this notation, $DT(W)_{exc}$ is the formal series obtained from $DT(W)$ by setting $u_{f_{e,0}}=0$ for all $e\in E_\cZ$:
\[
DT_{exc}(W):=DT(W)|_{u_f=0}
\]

\subsubsection{The crepant resolution conjecture}

Our main theorem is the following correspondence.

\begin{theorem}\label{thm:globalcrc}
With notation as above,
\[
DT_{mr}(\cZ)=\frac{DT(W)}{DT_{exc}(W)}
\]
after the change of variables $u_{g_{e,i}},u_{h_{e,i}}\rightarrow q_{e,i}$, $q\rightarrow q$, and
\[
u_{f_{e,0}}\rightarrow v_e\prod_{l=1}^{n_e-1}q_{e,l}^{(m_e+2)(n_e-l)}
\]

\end{theorem}

\begin{remark}
The change of variables given in Theorem \ref{thm:globalcrc} seems to depend, a priori, on the choice of orientation of each $e$. However, it is easy to check the independence of this choice.
\end{remark}

\subsection{The orbifold topological vertex}\label{sec:orbvertex}

In the particular case of toric targets, the DT potential has a beautiful combinatorial description, first developed by Okounkov--Reshitikhin--Vafa \cite{orv:qcycc} in the smooth case and later generalized by Bryan--Cadman--Young \cite{bcy:otv} to the orbifold setting.

Intuitively, the combinatorial description of the DT generating series follows from the simple fact that the Euler characteristic of the Hilbert scheme is equal to the sum to its torus-fixed points. A torus-fixed point of the Hilbert scheme is completely determined by its behavior at each torus-fixed point of $\cZ$, where it is locally defined by a monomial ideal with one-dimensional support. Such monomial ideals, in turn, correspond to asymptotic 3D partitions, as depicted in Figure \ref{monideal}. Therefore, the study of DT invariants can intuitively be reduced to the study of 3D partitions along with a suitable gluing algorithm for patching the corresponding subschemes together.

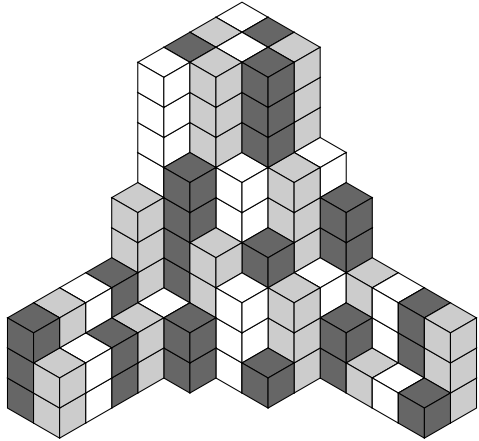
\begin{figure}
\begin{tikzpicture}[scale=.4]

\cube{0}{0}{0}{white}
\cube{0}{1}{0}{white}
\cube{0}{2}{0}{white}
\cube{0}{3}{0}{white}
\cube{0}{4}{0}{white}
\cube{0}{5}{0}{white}
\cube{0}{6}{0}{white}
\cube{0}{7}{0}{white}
\cube{0}{8}{0}{white}
%\continue{0}{11}{0}

\cube{0}{0}{1}{dgray}
\cube{0}{1}{1}{dgray}
\cube{0}{2}{1}{dgray}
\cube{0}{3}{1}{dgray}
\cube{0}{4}{1}{dgray}
\cube{0}{5}{1}{dgray}
\cube{0}{6}{1}{dgray}
\cube{0}{7}{1}{dgray}
\cube{0}{8}{1}{dgray}
%\continue{0}{11}{1}

\cube{0}{0}{2}{lgray}
\cube{0}{1}{2}{lgray}
\cube{0}{2}{2}{lgray}
\cube{0}{3}{2}{lgray}
\cube{0}{4}{2}{lgray}
\cube{0}{5}{2}{lgray}
\cube{0}{6}{2}{lgray}
\cube{0}{7}{2}{lgray}
\cube{0}{8}{2}{lgray}
%\continue{0}{11}{2}

\cube{0}{0}{3}{white}
\cube{0}{1}{3}{white}
\cube{0}{2}{3}{white}
\cube{0}{3}{3}{white}
\cube{0}{4}{3}{white}
\cube{0}{5}{3}{white}
%\cube{0}{6}{3}{white}
%\cube{0}{7}{3}{white}
%\cube{0}{8}{3}{white}
%\continue{0}{11}{3}

\cube{1}{0}{0}{lgray}
\cube{1}{1}{0}{lgray}
\cube{1}{2}{0}{lgray}
\cube{1}{3}{0}{lgray}
\cube{1}{4}{0}{lgray}
\cube{1}{5}{0}{lgray}
\cube{1}{6}{0}{lgray}
\cube{1}{7}{0}{lgray}
\cube{1}{8}{0}{lgray}
%\continue{1}{11}{0}

\cube{2}{0}{0}{dgray}
\cube{2}{1}{0}{dgray}
\cube{2}{2}{0}{dgray}
\cube{2}{3}{0}{dgray}
\cube{2}{4}{0}{dgray}
\cube{2}{5}{0}{dgray}
\cube{2}{6}{0}{dgray}
\cube{2}{7}{0}{dgray}
\cube{2}{8}{0}{dgray}
%\continue{2}{11}{0}

\cube{3}{0}{0}{white}
\cube{3}{1}{0}{white}
\cube{3}{2}{0}{white}
\cube{3}{3}{0}{white}
\cube{3}{4}{0}{white}
\cube{3}{5}{0}{white}
\cube{3}{6}{0}{white}
\cube{3}{7}{0}{white}
\cube{3}{8}{0}{white}
%\continue{3}{11}{0}

\cube{1}{0}{1}{white}
\cube{1}{1}{1}{white}
\cube{1}{2}{1}{white}
\cube{1}{3}{1}{white}
\cube{1}{4}{1}{white}
\cube{1}{5}{1}{white}
\cube{1}{6}{1}{white}
\cube{1}{7}{1}{white}
\cube{1}{8}{1}{white}
%\continue{1}{11}{1}

\cube{2}{0}{1}{lgray}
\cube{2}{1}{1}{lgray}
\cube{2}{2}{1}{lgray}
\cube{2}{3}{1}{lgray}
\cube{2}{4}{1}{lgray}
\cube{2}{5}{1}{lgray}
\cube{2}{6}{1}{lgray}
\cube{2}{7}{1}{lgray}
\cube{2}{8}{1}{lgray}
%\continue{2}{11}{1}

\cube{1}{0}{2}{dgray}
\cube{1}{1}{2}{dgray}
\cube{1}{2}{2}{dgray}
\cube{1}{3}{2}{dgray}
\cube{1}{4}{2}{dgray}
\cube{1}{5}{2}{dgray}
%\cube{1}{6}{2}{dgray}
%\cube{1}{7}{2}{dgray}
%\continue{1}{11}{2}

\cube{2}{0}{2}{white}
\cube{2}{1}{2}{white}
\cube{2}{2}{2}{white}
\cube{2}{3}{2}{white}
\cube{2}{4}{2}{white}
\cube{2}{5}{2}{white}
%\cube{2}{6}{2}{white}
%\cube{2}{7}{2}{white}
%\continue{2}{11}{2}

\cube{0}{0}{4}{dgray}
\cube{0}{1}{4}{dgray}
\cube{0}{2}{4}{dgray}
\cube{0}{3}{4}{dgray}
\cube{0}{4}{4}{dgray}

\cube{1}{0}{3}{lgray}
\cube{1}{1}{3}{lgray}
\cube{1}{2}{3}{lgray}
\cube{1}{3}{3}{lgray}
\cube{1}{4}{3}{lgray}
\cube{1}{5}{3}{lgray}

\cube{1}{6}{2}{dgray}
\cube{1}{7}{2}{dgray}
\cube{1}{8}{2}{dgray}

\cube{1}{0}{4}{white}
\cube{1}{1}{4}{white}
\cube{1}{2}{4}{white}

\cube{2}{0}{3}{dgray}
\cube{2}{1}{3}{dgray}
\cube{2}{2}{3}{dgray}
\cube{2}{3}{3}{dgray}

\cube{3}{0}{1}{dgray}
\cube{3}{1}{1}{dgray}
\cube{3}{2}{1}{dgray}
\cube{3}{3}{1}{dgray}
\cube{3}{4}{1}{dgray}
\cube{3}{5}{1}{dgray}

\cube{3}{0}{2}{lgray}
\cube{3}{1}{2}{lgray}
\cube{3}{2}{2}{lgray}
\cube{3}{3}{2}{lgray}

\cube{4}{0}{0}{lgray}
\cube{4}{1}{0}{lgray}
\cube{4}{2}{0}{lgray}
\cube{4}{3}{0}{lgray}
\cube{4}{4}{0}{lgray}

\cube{0}{0}{5}{lgray}
\cube{0}{1}{5}{lgray}
\cube{0}{2}{5}{lgray}

\cube{0}{0}{6}{white}

\cube{0}{0}{7}{dgray}

\cube{1}{0}{5}{dgray}
\cube{1}{1}{5}{dgray}

\cube{1}{0}{6}{lgray}

\cube{2}{0}{4}{lgray}
\cube{2}{1}{4}{lgray}
\cube{2}{2}{4}{lgray}

\cube{3}{0}{3}{white}
\cube{3}{1}{3}{white}
\cube{3}{2}{3}{white}

\cube{4}{0}{1}{white}
\cube{4}{1}{1}{white}

\cube{4}{0}{2}{dgray}
\cube{4}{1}{2}{dgray}

\cube{5}{0}{0}{dgray}
\cube{5}{1}{0}{dgray}
\cube{5}{2}{0}{dgray}

\cube{3}{0}{4}{dgray}

\cube{6}{0}{0}{white}
\cube{6}{1}{0}{white}

\cube{7}{0}{0}{lgray}

\cube{5}{0}{1}{lgray}

\cube{0}{1}{6}{white}
\cube{0}{2}{6}{white}
\cube{0}{1}{7}{dgray}
\cube{0}{2}{7}{dgray}
\cube{1}{0}{7}{white}
\cube{0}{0}{8}{lgray}
\cube{0}{1}{8}{lgray}
\cube{0}{2}{8}{lgray}
\cube{1}{0}{8}{dgray}

\cube{6}{0}{1}{dgray}
\cube{5}{1}{1}{lgray}
\cube{6}{1}{1}{dgray}
\cube{7}{1}{0}{lgray}
\cube{7}{0}{1}{white}
\cube{7}{1}{1}{white}
\cube{8}{0}{0}{dgray}
\cube{8}{1}{0}{dgray}
\cube{8}{0}{1}{lgray}
\cube{8}{1}{1}{lgray}

\cube{7}{3}{1}{white}
\cube{8}{3}{1}{lgray}
\cube{9}{3}{1}{dgray}

%\continue{0}{11}{4}
%\continue{1}{11}{3}
%\continue{1}{11}{4}
%\continue{3}{11}{2}
%\continue{3}{11}{1}
%\continue{4}{11}{1}
%\continue{4}{11}{0}

\end{tikzpicture}
\caption{This image depicts an asymptotic 3D partition $\Pi$. The three legs continue to infinity and are described by a triple of 2-d partitions. There is a corresponding monomial ideal in $\bC[x,y,z]$ generated by all monomials $x^iy^jz^k$ such that the box in the position $(i,j,k)$ does not belong to $\Pi$ (the indices, here, correspond to the back corner of the box). The boxes in $\Pi$ have been colored to reflect the $\bZ_3$ action on $\bC^3$ with weights $(1,-1,0)$, i.e. the $A_2$ singularity.}\label{monideal}
\end{figure}

This intuitive description for computing Euler characteristics is not quite right, since actual DT invariants are defined in terms of \emph{weighted} Euler characteristics \eqref{eqn:dtdef}. However, Bryan--Cadman--Young proved that incorporating the Behrend function merely amounts to a sign, which can easily be incorporated into the generating series. Interestingly, the results of the current paper hold true with or without the sign.

Rather than define the $A_{n-1}$ orbifold topological vertex explicitly in terms of generating functions of colored 3D partitions, we define it directly in terms of vertex operators, following Section 7 of \cite{bcy:otv}, as this will be the most useful perspective for the purposes of this paper. For the perspective of 3D partitions, we direct the reader to Section 3 of \cite{bcy:otv}. Before we can get to the definition, we must introduce some more notation.

\subsubsection{Partitions}
Let $\rho$ denote a partition, i.e. a non-increasing sequence of non-negative integers. We think of $\rho$ as a Young diagram in English notation where we index the rows and columns beginning with $0$. As an example of our conventions, the partition $\rho=(3,3,2,1,0,\dots)$ corresponds to the Young diagram
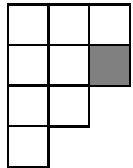
\begin{center}
\begin{ytableau}
*(white) & *(white) & *(white)\\
*(white) & *(white) & *(gray)\\
*(white) & *(white)\\
*(white)
\end{ytableau}
\end{center}
where the index of the shaded box is $(i,j)=(1,2)$. We define the \emph{size} of $\rho$, denoted $|\rho|$, to be the number of boxes in the corresponding Young diagram. We also define the $l$-th \emph{diagonal} of $\rho$ to be the set of boxes satisfying $j-i=l$. The conjugate partition $\rho'$ is obtained by reflecting the Young diagram along the $0$-th diagonal. 

Let $\bar\lambda$ be an $n$-colored Young diagram where the boxes in the $l$-th diagonal are colored by $(l$ mod $n)\in \{0,\dots,n-1\}$. For example, if $n=3$ we color the partition $\bar\lambda=(3,3,2,1,0,\dots)$ as follows:
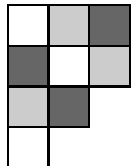
\[
\begin{ytableau}
*(white) & *(lgray) & *(dgray) \\
*(dgray) & *(white) & *(lgray) \\
*(lgray)  & *(dgray) \\
*(white)
\end{ytableau}
\]

An \emph{$n$-strip} of $\rho$ is a connected collection of $n$ boxes in the Young diagram that does not contain any $2\times 2$ squares, and a \text{border strip} is such a collection that lies entirely along the southeast border of the Young diagram. We define the \emph{height}, $ht(\nu)$, of a border strip $\nu$ to be the number of rows occupied by the strip, minus one.

In this paper, we only consider colored Young diagrams $\bar\lambda$ that are \emph{balanced} in the sense that they have the same number of boxes of each color. This restriction corresponds to the multi-regular specialization of the DT invariants. Each such $\bar\lambda$ can be decomposed (non-uniquely) by successively pulling off a sequence of $n$-border strips $(\nu_1,...\nu_{|\bar\lambda|/n})$. We define the quantity
\[
\frac{\chi_{\bar\lambda}(n^d)}{\dim(\lambda)}:=(-1)^{\sum_i ht(\nu_i)}=\pm 1
\]
where the notation on the left-hand side originates from an interpretation  in terms of the representation theory of the generalized symmetric group (see, for example, \cite{rz:ggmv} Section 6). It is easily checked that this quantity is independent of the choice of border strip decomposition.

\begin{remark}\label{rmk:young}
We warn the reader that there are two standard conventions concerning the representation of partitions as Young diagrams: English and French notation. We use English notation here. As Macdonald points out in his classic reference, those preferring French notation can read the Young diagrams ``upside down and in a mirror'' \cite{m:sfhp}. English notation is standard in the study of Schur functions, while French notation is more natural for vertex operators.
\end{remark}

\subsubsection{Vertex operators}

One of the key combinatorial tools in studying 3D partitions are the so-called \emph{vertex operators}. Following Section 7 of \cite{bcy:otv}, we give a concise review of the vertex operators used in the description of the $A_{n-1}$ orbifold topological vertex. 

Let $\mathcal{P}$ be the set of all partitions, $\mathcal{R}$ the space of formal Laurent series in formal variables $q_0,\dots,q_{n-1}$, and $\cRP$ the free $\mathcal{R}$-module generated over $\mathcal{P}$. Vertex operators are defined to act on the space $\cRP$. For two partitions $\tau$ and $\sigma$, we write $\tau\succ\sigma$ if
\[
\tau_0\geq \sigma_0\geq \tau_1\geq \sigma_1\geq\dots
\]
Geometrically, $\tau\succ\sigma$ if and only if $\tau'\supseteq\sigma'$ occur as consecutive diagonal slices in a 3D partition (the appearance of the conjugates here is a result of Remark \ref{rmk:young}).

For $x$ a monomial in $q_i$, define the vertex operators $\Gamma_{\pm 1}$ and $Q_k$ by their actions:
\begin{align*}
\Gamma_{+1}(x)\tau&:=\sum_{\sigma\prec\tau}x^{|\tau|-|\sigma|}\sigma\\
\Gamma_{-1}(x)\tau&:=\sum_{\sigma\succ\tau}x^{|\sigma|-|\tau|}\sigma\\
Q_k\tau&:=q_k^{|\tau|}\tau
\end{align*}
For $O$ an operator on $\cRP$, we define the \emph{expectation} $\langle\sigma|O|\tau\rangle$ to be the coefficient of $\sigma$ after applying the operator $O$ to $\tau$.

We will need the following important commutation relations.
\begin{lemma}
For $i,j=\pm 1$:
\begin{equation}\label{eqn:commutation}
\Gamma_i(a)\Gamma_j(b)=(1-ab)^{\frac{j-i}{2}}\Gamma_j(b)\Gamma_i(a)
\end{equation}
and
\begin{equation}\label{eqn:commutation2}
\Gamma_j(a)Q_k=Q_k\Gamma_j(aq_k^{j}).
\end{equation}
\end{lemma}
The interested reader can find proofs of these identities in \cite{bcy:otv}, Lemmas 29 and 30.

\subsubsection{Slope sequences}\label{slopesequences} Let $\bar\lambda(t)$ denote the \emph{slope sequence} of $\bar\lambda$. It is defined by setting 
\[
S(\bar\lambda)=:\{\bar\lambda_0-1,\bar\lambda_1-2,\bar\lambda_2-3,\dots\}
\]
and
\[
\bar\lambda(t):=\begin{cases}
+1 & t\in S(\bar\lambda)\\
-1 & t\notin S(\bar\lambda).
\end{cases}
\]

The relevance of the slope sequence is that it describes the boundary of the Young diagram $\bar\lambda$ (after rotating clockwise by $\pi/4$), see Figure \ref{fig:slope}. 

%\begin{remark}
% We warn the reader that our conventions are slightly different than \cite{bcy:otv} because our Young diagrams are positioned in English notation whereas theirs are in French. 
%\end{remark}

\begin{figure}
\hspace{-1.5cm}\begin{tikzpicture}
\begin{scope}[yshift=3.6cm,gray, very thin, scale=.6,]
\clip (-5.5, 0) rectangle (5.5,-5.5);
\draw[rotate=225, scale=1.412] (0,0) grid (6,6);
\end{scope}
\begin{scope}[yshift=3.6cm,rotate=45, very thick, scale=.6*1.412]
\draw (-0,-5.5) -- (-0, -3) -- (-1,-3) -- (-1,-3)-- (-2,-3) -- (-2,-2) -- (-3,-2) -- (-3,-1) -- (-4,-1) -- (-4,-0) -- (-5.5,-0);
\end{scope}
\begin{scope}[scale=.6, dotted,yshift=6cm]
\draw (4.5,0) -- (4.5, -4.5);
\draw (3.5,0) -- (3.5, -3.5);
\draw (2.5,0) -- (2.5,-3.5);
\draw (1.5,0) -- (1.5, -4.5);
\draw (.5,0) -- (.5, -4.5);
\draw (-.5,0) -- (-.5, -4.5);
\draw (-1.5,0) -- (-1.5, -4.5);
\draw (-2.5,0) -- (-2.5, -4.5);
\draw (-3.5,0) -- (-3.5, -4.5);
\draw (-4.5,0) -- (-4.5, -4.5);
\end{scope}
\begin{scope}[scale=.6, yshift=6.5cm]
\draw (-7.5,0)  node{$\bar\lambda(t)=$} node[above=3pt]{$t=$};
\draw (-5.5,0) node{$\cdots$};
\draw (-4.5,0)  node{$+1$} node[left=2pt,above=3pt]{$-5$};
\draw (-3.5,0)  node{$-1$} node[left=2pt,above=3pt]{$-4$};
\draw (-2.5,0)  node{$+1$} node[left=2pt,above=3pt]{$-3$};
\draw (-1.5,0)  node{$-1$} node[left=2pt,above=3pt]{$-2$};
\draw (-0.5,0)  node{$+1$} node[left=2pt,above=3pt]{$-1$};
\draw (0.5,0)  node{$-1$} node[above=3pt]{$0$};
\draw (1.5,0)  node{$+1$} node[above=3pt]{$1$};
\draw (2.5,0)  node{$+1$} node[above=3pt]{$2$};
\draw (3.5,0)  node{$-1$} node[above=3pt]{$3$};
\draw (4.5,0)  node{$-1$} node[above=3pt]{$4$};
\draw (5.5,0) node{$\cdots$};
\end{scope}
\end{tikzpicture}
\caption{Slope sequence for the partition $\bar\lambda=(3,3,2,1)$.}\label{fig:slope}
\end{figure}
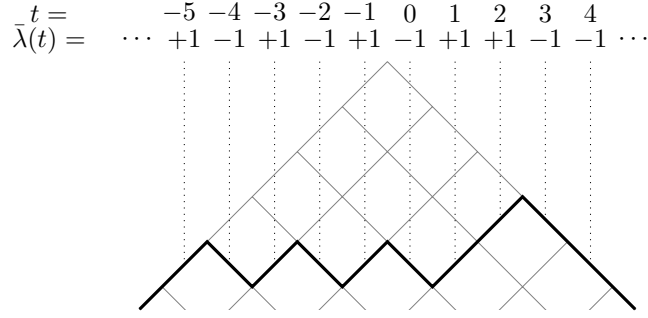

\subsubsection{The orbifold vertex} We are now ready to define the $A_{n-1}$ orbifold topological vertex. Let $\bq$ denote the variables $(q_0,\dots,q_{n-1})$ with indices computed modulo $n$ and $q:=q_0\cdots q_{n-1}$. Define the variables $\fq_t$ recursively by the rules $\fq_0=1$ and
\[
\fq_t=q_t\fq_{t-1}
\]
so that
\[
(\dots, \fq_{-2},\fq_{-1},\fq_0,\fq_1,\fq_2,\dots)=(\dots,q_{0}^{-1}q_{-1}^{-1},q_0^{-1},1,q_1,q_1q_2,\dots).
\]

\begin{definition}[c.f. \cite{bcy:otv}, Proposition 8]\label{lem:operators}
The \emph{reduced, multi-regular DT vertex for the $A_{n-1}$ vertex $\cX=[\bC^3/\bZ_n]$} is defined by
\[
P^{n}_{\rho^+,\rho^-,\bar\lambda}(\bq)=\frac{V^{\cX}_{\rho^+,\rho^-,\bar\lambda}(\bq)}{V^{\cX}_{\emptyset,\emptyset,\emptyset}(\bq)}
\]
where $V$ is the vertex operator expectation
\[
V^{n}_{\rho^+,\rho^-,\bar\lambda}(\bq):=q_0^{-|\rho^+|}\left\langle(\rho^-)'\Bigg| \prod_{t\in\bZ}^{\rightarrow}\Gamma_{\bar\lambda(t)}\left(\fq_t^{-\bar\lambda(t)} \right) \Bigg| \rho^+ \right\rangle
\]
where the arrow in the expectation denotes that the index $t$ is increasing from left to right.
\end{definition}

As proved in \cite{bcy:otv}, the formal series $P^{n}_{\rho^+,\rho^-,\bar\lambda}(\bq)$ is a suitably normalized generating function of 3D partitions, as in Figure \ref{monideal}, with asymptotic partitions given by $\rho^-$, $(\rho^+)'$, and $\bar\lambda'$ (again, Remark \ref{rmk:young} explains the conjugates).

\subsubsection{Loop Schur functions}\label{loopschur}

The $A_{n-1}$ orbifold topological vertex $P^{n}_{\rho^+,\rho^-,\bar\lambda}(\bq)$ can also be written in terms of loop Schur functions. In our related works \cite{rz:ggmv,r:cmn,rz:chialsf,r:gwdtcrc}, we heavily rely on the algebro-combinatorial structure of this formula. For completeness, we reproduce the formula here.

For the colored Young diagram $\bar\lambda=((\bar\lambda)_0\geq(\bar\lambda)_1\geq(\bar\lambda)_2\geq\dots)$, we define the sequence of variables $\fq_{\bullet-\bar\lambda}$ by
\[
\fq_{\bullet-\bar\lambda}:=(\fq_{-(\bar\lambda)_0},\fq_{1-(\bar\lambda)_1},\fq_{2-(\bar\lambda)_2},\dots).
\]
We set $\fq_\bullet:=\fq_{\bullet-\emptyset}$ and we use an overline on an expression in the $q$ variables to denote the exchange $q_i\leftrightarrow q_{-i}$. 

\begin{theorem}[\cite{bcy:otv}, Theorem 12]
The reduced Donaldson-Thomas vertex for $\cX$ can be written
\[
P^n_{\rho^+,\rho^-,\bar\lambda}(\bq)=\left(\prod_{(i,j)\in\bar\lambda}q_{j-i}^{i}\right)s_{\bar\lambda}(\bq)\sum_{\omega}q_0^{-|\omega|}\overline{s_{\rho^+/\omega}(\fq_{\bullet-\bar\lambda})}s_{(\rho^-)'/\omega}(\fq_{\bullet-\bar\lambda'})
\]
where $s_{\bar\lambda}(\bq)$ denotes the loop Schur function of $\bar\lambda$ in the variables $(q_0,\dots,q_{n-1})$ and $s_{\rho/\omega}$ denotes a usual skew Schur function.
\end{theorem}

\begin{remark}
Skew Schur functions are classical and the standard reference is Macdonald \cite{m:sfhp}. Loop Schur functions were introduced by Lam--Pylyavskyy \cite{lp:tpilg1} in the context of total positivity of matrix loop groups. References for loop Schur functions are a paper of Lam \cite{l:lsfafmp}, and more closely related to the topic at hand, a paper of the author \cite{r:cmn}. The series $s_{\bar\lambda}(\bq)$ are obtained from those defined in \cite{r:cmn} by specializing $x_{i,j}\rightarrow q_i^j$, and are expressible as rational functions in $\bq$.
\end{remark}

\subsection{Gluing formula}\label{sec:gluing}

The gluing algorithm of Bryan--Cadman--Young \cite{bcy:otv} describes how to recover the DT potential of $\cZ$ from the $A_{n-1}$ DT vertex. As it will be necessary below, we recall the algorithm here.

Let $\Lambda(\cZ)$ denote the set of edge assignments $\Lambda=\{\bar\lambda(e)\}_{e\in E_\cZ}$ where each $\bar\lambda(e)$ is a balanced $n_e$-colored partition. For an edge assignment $\Lambda$ and a vertex $v$, set $\bar\lambda(v)=(\bar\lambda_1(v),\bar\lambda_2(v),\bar\lambda_3(v))$ where
\[
\bar\lambda_i(v)=\begin{cases}
\bar\lambda(e_i(v)) & \text{ if } e_i(v) \text{ is outgoing}\\
\bar\lambda(e_i(v))' & \text{ if } e_i(v) \text{ is incoming}
\end{cases}
\]
and for each vertex $v$, set
\[
\bq_v=\begin{cases}
\bq_{e_3(v)} & \text{ if } e_3(v) \text{ is outgoing}\\
\overline{\bq_{e_3(v)}} & \text{ if } e_3(v) \text{ is incoming}
\end{cases}
\]
Theorem 10 of Bryan--Cadman--Young \cite{bcy:otv} can be restated as follows.

\begin{theorem}[\cite{bcy:otv}, Theorem 10]\label{thm:gluing}
Define 
\[
\underline{DT}(\cZ):=\sum_\Lambda\prod_e E_{\lambda_e}^{\cZ,e} \prod_v P^{n_{e_v}}_{\Lambda(v)}(\bq_v)
\]
where 
\[
E_{\lambda}^{\cZ,e}:=v_e^{|\bar\lambda|/n_e}(-1)^{(m_e+\delta_{e,0}+\delta_{e,\infty})|\bar\lambda|}\prod_{(i,j)\in\bar\lambda}q_{e,j-i}^{-m_ej-m_e'i+1}
\]
Then the reduced, multi-regular DT partition $DT(\cZ)$ is obtained from $\underline{DT}(\cZ)$ by adding a minus sign to the variables $q_{e,0}$ (and hence also to $q$).
\end{theorem}

\subsubsection{Resolutions}

Let $Y$ be the toric resolution of $\cX$. Then $Y$ contains a chain of $n-1$ $\bP^1$s, all of which have normal bundle $\cO\oplus\cO(-2)$. On each compact edge in the web diagram, corresponding to one of the $\bP^1$s, choose the orientation for which $N_r=\cO$. Orient each noncompact edge outward. At each of the $n$ vertices, choose $e_3(v)$ to be the edge corresponding to the fiber in the trivial direction. Then it is easy to check that the edge term in Theorem \ref{thm:gluing} becomes
\[
E_{\lambda_e}=(qv_e)^{|\lambda_e|}
\]
where $\lambda_e$ is a usual partition.

Therefore, we can write 
\begin{align}\label{eqn:closeddty}
\nonumber\underline{DT}(Y)&=\sum_{\tau_1,\dots,\tau_{n-1}} P^1_{\tau_1,\emptyset,\emptyset}(\bq)(qv_{1})^{|\tau_{1}|}P^1_{\tau_{2},\tau_{1}',\emptyset}(\bq)(qv_{2})^{|\tau_{2}|}\\
&\cdots (qv_{n-2})^{|\tau_{n-2}|}P^1_{\tau_{n-1},\tau_{n-2}',\emptyset}(\bq)(qv_{n-1})^{|\tau_{n-1}|}P^1_{\emptyset,\tau_{n-1}',\emptyset}(\bq)
\end{align}

To obtain the analog of $P^{\cX}_{\rho^+,\rho^-,\bar\lambda}(\bq)$ on the resolution, we generalize \eqref{eqn:closeddty} by modifying the vertex terms:
\begin{align}\label{eqn:opendty}
\nonumber P^Y_{\rho^+,\rho^-,(\lambda_0,\dots,\lambda_{n-1})}&(q,\bv):=\sum_{\tau_1,\dots,\tau_{n-1}} P^1_{\tau_1,\rho^-,\lambda_0}(q)(qv_1)^{|\tau_1|}P^1_{\tau_2,\tau_1',\lambda_1}(q)(qv_2)^{|\tau_2|}\\
&\cdots (qv_{n-2})^{|\tau_{n-2}|}P^1_{\tau_{n-1},\tau_{n-2}',\lambda_{n-2}}(q)(qv_{n-1})^{|\tau_{n-1}|}P^1_{\rho^+,\tau_{n-1}',\lambda_{n-1}}(q)
\end{align}
%Again, for later convenience we define a slight modification
%\begin{equation}\label{eqn:opendty2}
%\hat P^Y_{\rho^+,\rho^-,\lambda}(q,\bv):=\left(\prod_k\prod_{(i,j)\in\lambda_k}q^{-i}\right)P^Y_{\rho^+,\rho^-,\lambda}(q,\bv)
%\end{equation}
%which can also be defined by hatting all of the $P$s in \eqref{eqn:opendty}.
%

\section{Vertex crepant resolution conjecture}\label{sec:vertexcrc}

The crepant resolution conjecture for the DT orbifold topological vertex is most easily written in terms of the normalized series:
\[
\tilde P^{\cX}_{\rho^+,\rho^-,\bar\lambda}(\bq):=\left(\frac{\chi_{\bar\lambda}(n^d)}{\dim(\lambda)}\prod_{(i,j)\in\bar\lambda}q_{j-i}^{i}\right)P^{\cX}_{\rho^+,\rho^-,\bar\lambda}(\bq)
\]
and
\[
\tilde P^Y_{\rho^+,\rho^-,(\lambda_0,\dots,\lambda_{n-1})}:= \left(\prod_k\prod_{(i,j)\in\lambda_k}q^{i}\right)P^Y_{\rho^+,\rho^-,(\lambda_0,\dots,\lambda_{n-1})}.
\]

The following theorem is the technical heart of this paper.

\begin{theorem}\label{thm:vertexcrc}
After the change of variables $v_i\leftrightarrow q_i$,
\[
\tilde P^{\cX}_{\rho^+,\rho^-,\bar\lambda}(\bq)=\frac{\tilde P^Y_{\rho^+,\rho^-,\lambda}(q,\bv)}{\tilde P^Y_{\emptyset,\emptyset,\emptyset}(q,\bv)}\left(\prod_{(i,j)\in\lambda_k}\left((-1)^{n-k-1}q^{(n-k-1)(i-j)}\prod_{l> k}q_l^{n-l} \right)\right)
\]
where $\lambda=(\lambda_0,\dots,\lambda_{n-1})$ is the $n$-quotient of $\bar\lambda$ (see Section \ref{sec:quotients} below for a review of $n$-quotients). 
\end{theorem}

The rest of Section \ref{sec:vertexcrc} is devoted to the proof of Theorem \ref{thm:vertexcrc}.

\subsection{Quotients}\label{sec:quotients}

As it is essential in both the statement and the proof of Theorem \ref{thm:vertexcrc}, we briefly recall the basic definitions from the theory of $n$-quotients. From $\lambda=(\lambda_0,\dots,\lambda_{n-1})$, we obtain $n$ slope sequences $(\lambda_0(t),\dots,\lambda_{n-1}(t))$ as described in \ref{slopesequences}. From these $n$ slope sequences, we define a new slope sequence
\begin{equation}\label{eqn:quotients}
\bar\lambda(t):=\lambda_{\underline{t}}\left(\frac{t-\underline{t}}{n}\right)
\end{equation}
where $\underline{t}:=t$ mod $n$. Concretely, all we are doing is simply interlacing the slope sequences $(\lambda_0(t),\dots,\lambda_{n-1}(t))$.

\begin{definition}
If the slope sequences for $\lambda$ and $\bar\lambda$ are related as in \eqref{eqn:quotients}, we say that $\lambda$ is the \emph{$n$-quotient} of $\bar\lambda$. 
\end{definition}

\begin{remark}
Usually when one speaks of $n$-quotients, they also mention $n$-cores. However, the balanced partitions we consider in this paper are characterized by the property that they have empty $n$-core.
\end{remark}

It will be helpful later to have a better understanding of the $n$-quotient correspondence in terms of Young diagrams, rather than slope sequences. In particular, we make the following observation which can easily be checked.

\begin{observation}
Adding a single box to the $(i,j)$ position of $\lambda_l$ corresponds to adding a length $n$ border strip to the colored partition $\bar\lambda$. This border strip has exactly one box of each color, the northeastern-most box has color $l$ and the unique color $0$ box lies in the $n(j-i)$ diagonal of $\bar\lambda$.
\end{observation}

\begin{example}\label{ex:addbox}
Suppose $n=5$, $\lambda_0=(1)$, $\lambda_1=\emptyset$, $\lambda_2=(2,1)$, $\lambda_3=(2)$, and $\lambda_4=\emptyset$. Then one checks that $\bar\lambda$ is the Young diagram
\[
\begin{ytableau}
0 & 1 & 2 & 3 & 4 & 0 & 1 & 2 & 3\\
4 & 0 & 1 & 2 & 3 & 4 & 0 & 1 & 2\\
3 & 4 & 0\\
2 & 3 & 4\\
1 & 2\\
0 & 1\\
4\\
3\\
\end{ytableau}
\]
where the numbers denote the colors. Now if we add a box in the $(1,1)$ position of $\lambda_2$, the colored Young diagram becomes
\[
\begin{ytableau}
\circled{0} & \circled{1} & \circled{2} & 3 & 4 & 0 & 1 & 2 & 3\\
\circled{4} & 0 & 1 & 2 & 3 & 4 & 0 & 1 & 2\\
\circled{3} & 4 & 0 & *(lgray) 1 & *(lgray) 2\\
2 & 3 & 4 & *(lgray) 0\\
1 & 2 & *(lgray) 3 & *(lgray) 4\\
0 & 1\\
4\\
3\\
\end{ytableau}
\]
We have suggestively decorated the second Young diagram. The highlighted boxes are the new border strip. Rather than interpreting the modification in terms of simply adding these new boxes, a more useful interpretation is to add the boxes containing circles along the base and to push out the pre-existing boxes along their respective diagonals.
\end{example}

\subsection{Interlacing operators}

Consider the formal function $P^Y_{\rho^+,\rho^-,\lambda}(q,\bv)$ defined in \eqref{eqn:opendty}. By definition, we can write
\[
\frac{P^Y_{\rho^+,\rho^-,\lambda}(q,\bv)}{P^Y_{\emptyset,\emptyset,\emptyset}(q,\bv)}=\frac{V^Y_{\rho^+,\rho^-,\lambda}(q,\bv)}{V^Y_{\emptyset,\emptyset,\emptyset}(q,\bv)}
\]
where
\begin{align*}
V^Y_{\rho^+,\rho^-,\lambda}(q,\bv)=q^{-|\rho^+|}\sum_{\tau_1,\dots,\tau_{n-1}} &\left\langle(\rho^-)'\Bigg| \prod_{t\in\bZ}^{\rightarrow}\Gamma_{\lambda_0(t)}\left(q^{-t\lambda_0(t)} \right) \Bigg| \tau_1 \right\rangle v_1^{|\tau_1|}\\
&\cdot\left\langle\tau_1\Bigg| \prod_{t\in\bZ}^{\rightarrow}\Gamma_{\lambda_1(t)}\left(q^{-t\lambda_1(t)} \right) \Bigg| \tau_2 \right\rangle v_2^{|\tau_2|}\\
&\hspace{-.5cm}\cdots v_{n-1}^{|\tau_{n-1}|}\left\langle\tau_{n-1}\Bigg| \prod_{t\in\bZ}^{\rightarrow}\Gamma_{\lambda_{n-1}(t)}\left(q^{-t\lambda_{n-1}(t)} \right) \Bigg| \rho^+ \right\rangle
\end{align*}

We now perform the change of variables $v_i\rightarrow q_i$, approximate the infinite operator expressions with finite ones, and concatenate the expectations in the natural way. We obtain
\begin{align*}
V^Y_{\rho^+,\rho^-,\lambda}(\bq)=q^{-|\rho^+|}\lim_{N\rightarrow\infty}\Bigg\langle&(\rho^-)'\Bigg|\prod_{t}^{\rightarrow}\Gamma_{\lambda_0(t)}\left(q^{-t\lambda_0(t)} \right) Q_1 \prod_{t}^{\rightarrow}\Gamma_{\lambda_1(t)}\left(q^{-t\lambda_1(t)} \right) Q_2 \\
&\cdots Q_{n-1} \prod_{t}^{\rightarrow}\Gamma_{\lambda_{n-1}(t)}\left(q^{-t\lambda_{n-1}(t)} \right)\Bigg| \rho^+ \Bigg\rangle
\end{align*}
where the index $t$ satisfies $-N\leq t\leq N-1$. 

If we commute all of the $Q_k$ operators to the right, then we arrive at the following expression
\begin{align}\label{eqn:operatorexpectation}
\nonumber V^Y_{\rho^+,\rho^-,\lambda}(\bq)=q_0^{-|\rho^+|}\lim_{N\rightarrow\infty}\Bigg\langle&(\rho^-)'\Bigg|\prod_{t}^{\rightarrow}\Gamma_{\lambda_0(t)}\left(\fq_{nt}^{-\lambda_0(t)} \right) \prod_{t}^{\rightarrow}\Gamma_{\lambda_1(t)}\left(\fq_{nt+1}^{-\lambda_1(t)} \right) \\
&\cdots \prod_{t}^{\rightarrow}\Gamma_{\lambda_{n-1}(t)}\left(\fq_{nt+n-1}^{-\lambda_{n-1}(t)} \right)\Bigg| \rho^+ \Bigg\rangle
\end{align}
The final step is to interlace the operators appearing in expression \eqref{eqn:operatorexpectation} so that the indices on the $\fq$ variables are increasing from left to right. By definition, this interlacing of slope sequences produces the slope sequence for $\bar\lambda$. Therefore, we have
\[
V^Y_{\rho^+,\rho^-,\lambda}(\bq)=q_0^{-|\rho^+|}\lim_{N\rightarrow\infty}F_\lambda(N)\Bigg\langle(\rho^-)'\Bigg|\prod_{-nN\leq t\leq nN-1}^{\rightarrow}\Gamma_{\bar\lambda(t)}\left(\fq_t^{-\bar\lambda(t)}\right)\Bigg| \rho^+ \Bigg\rangle
\] 
where the factor $F_\lambda(N)$ arises from commuting the $\Gamma_{\pm}$ operators -- notice that this factor does not depend on $\rho^{\pm}$. The expectation on the right is simply the numerator in the vertex operator expression for $P^{\cX}_{\rho^+,\rho^-,\lambda}(\bq)$ in the limit $N\rightarrow\infty$. Therefore, in order to prove Theorem \ref{thm:vertexcrc}, it is left to analyze the limit
\[
\lim_{N\rightarrow\infty}\frac{F_\lambda(N)}{F_\emptyset(N)}.
\]
In particular, it suffices to prove the following
\begin{equation}\label{eqn:commuteidentity}
\prod_{(i,j)\in\lambda_k}\left((-1)^{n-k-1}q^{(n-k)(i-j)+j}\prod_{l> k}q_l^{n-l} \right)\lim_{N\rightarrow\infty}\frac{F_\lambda(N)}{F_\emptyset(N)}=\frac{\chi_{\bar\lambda}(n^d)}{\dim(\lambda)}\prod_{(i,j)\in\bar\lambda}q_{j-i}^i
\end{equation}
We prove \eqref{eqn:commuteidentity} inductively by systematically removing $n$-border strips from $\bar\lambda$. More specifically, let $\nu=(\nu_0,\dots,\nu_{n-1})$ be an $n$-tuple of partitions such that $\nu_i=\lambda_i$ for $i\neq k$ and $\lambda_k\setminus\nu_k$ is a single box in the $(i,j)$ position. Theorem \ref{thm:vertexcrc} follows from the following proposition.

\begin{proposition}\label{lem:commuteidentity}
\[
(-1)^{n-k-1}q^{(n-k)(i-j)+j}\prod_{l> k}q_l^{n-l}\lim_{N\rightarrow\infty}\frac{F_\lambda(N)}{F_\nu(N)}=(-1)^{ht(\bar\lambda\setminus\bar\nu)}\prod_{(r,s)\in\bar\lambda\setminus\bar\nu}q_{s-r}^r
\]
\end{proposition}

We prove Proposition \ref{lem:commuteidentity} in the next three subsections.

\subsection{Combinatorial description of $\lim_{N\rightarrow\infty}\frac{F_\lambda(N)}{F_\nu(N)}$ }
We start by closely analyzing the expression $\lim_{N\rightarrow\infty}\frac{F_\lambda(N)}{F_\nu(N)}$. Since we removed a box in the $(i,j)$ position of $\lambda_k$, the discrepancy in the operator expressions \eqref{eqn:operatorexpectation} for $\frac{V^Y_{\emptyset,\emptyset,\lambda}(q,\bv)}{V^Y_{\emptyset,\emptyset,\nu}(q,\bv)} $ can be expressed as the quotient
\begin{equation}\label{eqn:quotient}
\frac{\Gamma_-(q_1\cdots q_kq^{j-i-1})\Gamma_+(q_1^{-1}\cdots q_k^{-1}q^{i-j})}{\Gamma_+(q_1^{-1}\cdots q_k^{-1}q^{i-j+1})\Gamma_-(q_1\cdots q_kq^{j-i})}
\end{equation}
In other words, the factor $\lim_{N\rightarrow\infty}\frac{F_\lambda(N)}{F_\nu(N)}$ comes entirely from commuting $\Gamma_\pm$ operators through both the numerator and denominator of \eqref{eqn:quotient}.

For $l<k$, we must commute the operators $\Gamma_{\lambda_l(t)}\left(\fq_{nt+l}^{-\lambda_l(t)}\right)$ through \eqref{eqn:quotient} from left to right for all $t>j-i$ (and halfway for $t=j-i$). For $t>j-i$, we compute from \eqref{eqn:commutation} that this commutation produces the factor
\begin{equation}\label{eqn:factor1}
\begin{cases}
\frac{1-q_{l+1}\cdots q_k q^{j-i-t}}{1-q_{l+1}\cdots q_k q^{j-i-t-1}} &\text{ if } \lambda_l(t)=+1\\
\frac{1-q_{l+1}^{-1}\cdots q_k^{-1}q^{-j+i+t}}{1-q_{l+1}^{-1}\cdots q_k^{-1}q^{-j+i+t+1}} &\text{ if } \lambda_l(t)=-1
\end{cases}
\end{equation}
and for $t=j-i$ it produces the factor
\begin{equation}\label{eqn:factor2}
\begin{cases}
\left(1-q_{l+1}\cdots q_k q^{-1}\right)^{-1} &\text{ if } \lambda_l(t)=+1\\ 
\left(1-q_{l+1}^{-1}\cdots q_k^{-1}q\right)^{-1} &\text{ if } \lambda_l(t)=-1
\end{cases}
\end{equation}

For $t\gg0$, we are always in the second case of \eqref{eqn:factor1} by definition of the slope function. In particular, for $N$ sufficiently large a quick analysis shows that for $t=N-1,N-2,N-3,\dots$ the successive commutation factors in \eqref{eqn:factor1} cancel except for an initial term of $\left(1-q_{l+1}^{-1}\cdots q_k^{-1}q^{N-j+i} \right)^{-1}$ which tends to $1$ as $N\rightarrow\infty$. As we decrease $t$, the cancellation continues to happen until we encounter a place where $\lambda_l(t+1)=-1$ and $\lambda_l(t)=+1$ with $t\geq j-i$. At this point, the successive terms cancel modulo a multiplicative factor of
\begin{equation}\label{eqn:factor3}
-q_{l+1}^{-1}\cdots q_k^{-1}q^{-j+i+t+1}
\end{equation}
Similarly, whenever we encounter a place where $\lambda_l(t+1)=+1$ and $\lambda_l(t)=-1$ with $t\geq j-i$, we obtain a factor of
\begin{equation}\label{eqn:factor4}
-q_{l+1}\cdots q_kq^{j-i-t-1}.
\end{equation}

Since the factors \eqref{eqn:factor3} and \eqref{eqn:factor4} alternate, they merely contribute factors of $q$ except for possibly the last occurence. From this, it is not hard to see that the overall factor appearing is
\begin{equation}\label{eqn:totalfactor1}
\begin{cases}
-q_{l+1}^{-1}\cdots q_k^{-1}q^a &\text{ if } \lambda_l(j-i)=+1\\
q^a & \text{ if } \lambda_l(j-i)=-1
\end{cases}
\end{equation}
where $a:=\#\{t:t>j-i\text{ and }\lambda_l(t)=+1\}$. In terms of the Young diagram $\lambda_l$, we can describe the occurrence of each of the two cases combinatorially. In what follows, we use $d_l(\tau)$ to denote the number of boxes in the $l$-th diagonal of $\tau$. From the above analysis (and similar analysis for $l>k$) we conclude the following lemma.
\begin{lemma} 
\[
\lim_{N\rightarrow\infty}\frac{F_\lambda(N)}{F_\nu(N)}=\prod_{l\neq k} F_{\lambda\setminus\nu}(l)
\]
where $F_{\lambda\setminus\nu}(l)$ is defined by the following combinatorial rules.

\begin{enumerate}[(A)]
	\item If $l<k$, and\\
	\begin{enumerate}[(I)]
		\item $j-i<0$, then $F_{\lambda\setminus\nu}(l)$ is equal to\\
		\begin{enumerate}[(a)]
			\item $q^a$ if $d_{j-i}(\lambda_l)=a+j-i$ and $d_{j-i+1}(\lambda_l)=a+j-i+1$, and
			\item $-q_{l+1}^{-1}\cdots q_k^{-1}q^a$ if $d_{j-i}(\lambda_l)=d_{j-i+1}(\lambda_l)=a+j-i$.\\
		\end{enumerate}
		\item $j-i\geq 0$,  then $F_{\lambda\setminus\nu}(l)$ is equal to\\
		\begin{enumerate}[(a)]
			\item $q^a$ if the $d_{j-i}(\lambda_l)=d_{j-i+1}(\lambda_l)=a$, and
			\item $-q_{l+1}^{-1}\cdots q_k^{-1}q^{a+1}$ if $d_{j-i}(\lambda_l)=a+1$ and $d_{j-i+1}(\lambda_l)=a$.\\
		\end{enumerate}
	\end{enumerate}
	\item[(B)] If $l>k$ and\\
	\begin{enumerate}[(I)]
		\item $j-i\leq 0$, then $F_{\lambda\setminus\nu}(l)$ is equal to\\
		\begin{enumerate}[(a)]
			\item $q^a$ if $d_{j-i-1}(\lambda_l)=d_{j-i}(\lambda_l)=a$, and
			\item $-q_{k+1}^{-1}\cdots q_l^{-1}q^{a+1}$ if $d_{j-i-1}(\lambda_l)=a$ and $d_{j-i}(\lambda_l)=a+1$.\\
		\end{enumerate}
		\item $j-i>0$, then $F_{\lambda\setminus\nu}(l)$ is equal to\\
		\begin{enumerate}[(a)]
			\item $q^a$ if the $d_{j-i-1}(\lambda_l)=a-j+i+1$ and $d_{j-i}(\lambda_l)=a-j+i$, and
			\item $-q_{k+1}^{-1}\cdots q_l^{-1}q^a$ if $d_{j-i-1}(\lambda_l)=d_{j-i}(\lambda_l)=a-j+i$.\\
		\end{enumerate}
	\end{enumerate}
\end{enumerate}
\end{lemma}

Notice that the factor only ever depends on the lengths of the $(j-i)$th and $(j-i+1)$th diagonals if $l<k$ and the lenghts of the $(j-i)$th and $(j-i-1)$th diagonals if $l>k$. It is left to compare these factors with the other terms appearing in Proposition \ref{lem:commuteidentity}.

We proceed towards the proof of Proposition \ref{lem:commuteidentity} inductively in the next two subsections.

\subsection{Base case}

Let $\lambda[k]$ be the $n$-tuple of partitions obtained from $\lambda$ by setting $\lambda_l=\emptyset$ for $l\neq k$. Then $\overline{\lambda[k]}$ is a colored Young diagram which can be constructed entirely out of $n$-strips where the northeastern-most box of each strip has color $k$. 

\begin{example}\label{ex:kyoung}
Suppose $n=4$ and $\lambda_1=(3,3)$, then $\lambda_1$ and $\overline{\lambda[1]}$ correspond as follows:
\[
\lambda_1=\ytableaushort{
abc,def
}
\hspace{1cm}
\overline{\lambda[1]}=\ytableaushort{
{0_a}{1_a}{2_b}{3_b}{0_b}{1_b}{2_c}{3_c}{0_c}{1_c},{3_a}{0_e}{1_e}{2_f}{3_f}{0_f}{1_f},{2_a}{3_e},{1_d}{2_e},{0_d},{3_d},{2_d}
}
\]
where the lettering indicates which $4$-strips correspond to each box in $\lambda_1$.
\end{example}

We now show that Proposition \ref{lem:commuteidentity} holds for $\lambda[k]$. It is not hard to see (c.f. Example \ref{ex:kyoung}) that the right-hand side of Proposition \ref{lem:commuteidentity} in the case of removing the $(i,j)$ box from $\lambda[k]$ is equal to
\begin{equation}\label{eqn:kleft}
\begin{cases}
(-1)^{n-1}q^{n(i-j)-k+j}q_{k-1}q_{k-2}^2\cdots q_{k+1}^{n-1} &\text{ if } j-i<0\\
(-1)^{n-k-1}q^iq_{n-1}q_{n-2}^{2}\cdots q_{k+1}^{n-k-1}&\text{ if } j-i=0\\
q^i &\text{ if } j-i>0
\end{cases}
\end{equation}

It is also straightforward to compute the left-hand side of Proposition \ref{lem:commuteidentity} in the case of removing the $(i,j)$ box from $\lambda[k]$. For example, if we take $j-i<0$, then from the combinatorial rules (A) and (B) for $\lim_{N\rightarrow\infty}\frac{F_\lambda(N)}{F_\nu(N)}$ we have factors
\[
\begin{cases}
-q_{l+1}^{-1}\cdots q_k^{-1}q^{i-j} & \text{ if } l<k\\
1 &\text{ if } l>k
\end{cases}
\]
Multiplying these factors together and combining with the remaining terms in the left side of Proposition \ref{lem:commuteidentity}, we obtain exactly the first case of \eqref{eqn:kleft}. The other two cases are similar. This proves that Proposition \ref{lem:commuteidentity} holds for $\lambda[k]$.

\subsection{Inductive step}

To finish the proof of Proposition \ref{lem:commuteidentity}, we analyze what happens when we build $\bar\lambda$ from $\overline{\lambda[k]}$. By induction, suppose Proposition \ref{lem:commuteidentity} holds for $\bar\mu$ where $\bar\mu$ is a partial reconstruction of $\bar\lambda$, ie. $\mu_l\subset\lambda_l$ for $l\neq k$ and $\mu_k=\lambda_k$. We must show that the proposition continues to hold if we add one more box to $\lambda_l\setminus\mu_l$ for some $l$.

For simplicity, assume henceforth that $j-i<0$ and $l<k$, all other cases are similar. The only way that adding a box to $\mu_l$ affects the border strip corresponding to the $(i,j)$ box of $\lambda_k$ is if the additional box is in the $(j-i)$th or $(j-i+1)$th diagonal of $\mu_l$ (c.f. Example \ref{ex:addbox}). If the new box is in the $(j-i)$th diagonal of $\mu_l$, then the boxes of color $k+1,\dots,n-1,0,\dots,l$ in the strip corresponding to the $(i,j)$ box of $\lambda_k$ get shifted out. If the box is in the $(j-i+1)$th diagonal $\mu_l$, then the boxes of color $l+1,\dots,k$ get shifted out. We exhibit this behavior in the next example.

\begin{example}
Suppose $n=5$, $\mu_0=(1)$, $\mu_1=(1)$, $\mu_2=(2,1)$, $\mu_3=\emptyset$, $\mu_4=\emptyset$, $k=2$, and $(i,j)=(1,0)$. Then $\bar\mu$ is the colored Young diagram
 \[
\begin{ytableau}
0 & 1 & 2 & 3 & 4 & 0 & 1 & 2\\
4 & 0 & 1\\
3 & 4 & 0\\
2 & 3 & 4\\
1 & 2 & 3\\
*(lgray) 0 & *(lgray) 1 & *(lgray) 2\\
*(lgray) 4\\
*(lgray) 3
\end{ytableau}
\]
where we have shaded the special strip corresponding to the $(i,j)$ box of $\lambda_k$. Now if we add the $(1,0)$ box to the $-1$ diagonal of $\mu_1$ it has the following effect
 \[
\begin{ytableau}
0 & 1 & 2 & 3 & 4 & 0 & 1 & 2\\
4 & 0 & 1\\
3 & 4 & 0\\
2 & 3 & 4\\
\circled{1} & 2 & 3\\
\circled{0} & *(lgray) 1 & *(lgray) 2\\
\circled{4} & *(lgray) 0\\
\circled{3} & *(lgray) 4\\
\circled{2} & *(lgray) 3
\end{ytableau}
\]
where the boxes with circles are the strip corresponding to the new box in $\mu_1$ while the shaded boxes are the strip corresponding to the $(i,j)$ box of $\lambda_k$. We see that adding the box had the effect of shifting the color $3,4,0$ boxes along their diagonals. 
\end{example}

So what effect does this shift have on Proposition \ref{lem:commuteidentity}? On the right side of Proposition \ref{lem:commuteidentity} it is not hard to see that the shift results in a factor
\begin{equation}\label{eqn:lastfactor}
\begin{cases} 
-q_{k+1}\cdots q_{n-1}q_0\cdots q_l &\text{ if } d_{j-i}(\mu_l) \text{ increases}\\
-q_{l+1}\cdots q_k & \text{ if } d_{j-i+1}(\mu_l) \text{ increases}
\end{cases}
\end{equation}
The sign comes from the fact that the number of rows occupied by the strip is changing by exactly one. Notice that the only thing that changes on the left side is the factor $\lim_{N\rightarrow\infty}\frac{F_\lambda(N)}{F_\nu(N)}$. By the combinatorial description of this factor which we derived above, we see that if $d_{j-i}(\mu_l)$ increases then we must be passing from (A.I.b) to (A.I.a) (with an increase by one of the value $a$). But the discrepancy in these factors is $-q_{k+1}\cdots q_{n-1}q_0\cdots q_l$, agreeing with \eqref{eqn:lastfactor}. Similarly, if $d_{j-1+i}(\mu_l)$ increases then we must be passing from (A.I.a) to (A.I.b) (with the same value $a$). This results in a factor of $-q_{l+1}\cdots q_k$, agreeing again with \eqref{eqn:lastfactor}. 

A similiar check of the other cases finishes the proof of Proposition \ref{lem:commuteidentity} and, hence, Theorem \ref{thm:vertexcrc}. \vspace{-.55cm}\flushright $\square$

\section{Compatibility with gluing}\label{sec:globalcrc}

Having proved the vertex correspondence, we now turn to the task of proving that it is compatible with the edge terms in the gluing formula. We use the notation from Section \ref{sec:dt}.

\begin{proof}[Proof of Theorem \ref{thm:globalcrc}]

The edge terms in Theorem \ref{thm:gluing} naturally fall into two cases depending on whether an edge $e$ in the web diagram of $\cZ$ has $n_e=1$ or $n_e>1$. 

The case $n_e=1$ is relatively easy. Given such an oriented edge $e$ and inducing the same orientation and labeling of vertices for $f_e$, it is not hard to see that $m_{f_e}=m_e$,  $m_{f_e}'=m_e'$, and all $\delta$s agree. Therefore, for any partition $\rho$, the edge terms in Theorem \ref{thm:gluing} coincide and the contributions from such edges satisfy the statement of Theorem \ref{thm:globalcrc}.

The case $n_e>1$ is a bit more subtle. Fix an orientation on $e$ which induces an orientation on $f_k:=f_{e,k}$ (we drop the $e$ subscripts from this point on). Then it is not hard to compute that 
\[
m_{f_k}=nm+2(n-k-1)
\] 
and 
\[
m_{f_k}'=-nm-2(n-k)
\] 
We have implicitly chosen a labeling such that the $f_k$ are the third edge at each of their vertices, hence all of the $\delta$s are zero. Assume we have an edge assignment where the partition associated to each $f_k$ is denoted $\lambda_k$. Then locally at $e$ the contribution to the right side of Theorem \ref{thm:globalcrc} is given by
\begin{equation}\label{eqn:glueres}
\frac{P^Y_{\rho_0^+,\rho_0^-,(\lambda_0,\dots,\lambda_{n-1})}(q,\mathbf{u}_g)}{P^Y_{\emptyset,\emptyset,\emptyset}(q,\mathbf{u}_g)}E_{\lambda}^{W,e}\frac{P^Y_{\rho_\infty^-,\rho_\infty^+,(\lambda_{n-1}',\dots,\lambda_0')}(q,\overline{\mathbf{u}_h})}{P^Y_{\emptyset,\emptyset,\emptyset}(q,\overline{\mathbf{u}_h})}
\end{equation}
where
\[
E_\lambda^{W,e}=\prod_{(i,j)\in\lambda_k} u_{f_k}(-1)^{nm}q^{(nm+2(n-k))(i-j)+2j+1}
\]
and where $\mathbf{u}_g=(u_{g_1},\dots,u_{g_{n-1}})$ while $\overline{\mathbf{u}_h}=(u_{h_{n-1}},\dots,u_{h_1})$. After the change of variables, the edge factor becomes
\begin{equation}\label{eqn:edgefactor}
E_\lambda^{W,e}=\prod_{(i,j)\in\lambda_k}v\left(\prod_{l=1}^kq_l^{l-(m+1)l}\prod_{l=k+1}^{n-1}q_l^{(m+2)(n-l)}\right)(-1)^{nm}q^{(nm+2(n-k))(i-j)+2j+1}
\end{equation}
Applying Theorem \ref{thm:vertexcrc}, we have (after the change of variables)
\begin{equation}\label{eqn:vertex1}
\frac{ P^Y_{\rho_0^+,\rho_0^-,(\lambda_0,\dots,\lambda_{n-1})}(q,\mathbf{u}_g)}{ P^Y_{\emptyset,\emptyset,\emptyset}(q,\mathbf{u}_g)}=\tilde P^{\cX}_{\rho_0^+,\rho_0^-,\bar\lambda}(\bq)\prod_{(i,j)\in\lambda_k}\left((-1)^{n-k-1}q^{(n-k)(j-i)-j}\prod_{l > k}q_l^{l-n} \right)
\end{equation}
and
\begin{equation}\label{eqn:vertex2}
\frac{ P^Y_{\rho_\infty^-,\rho_\infty^+,(\lambda_{n-1}',\dots,\lambda_0')}(q,\overline{\mathbf{u}_h})}{ P^Y_{\emptyset,\emptyset,\emptyset}(q,\overline{\mathbf{u}_h})}=\tilde P^{\cX}_{\rho_\infty^-,\rho_\infty^+,\bar\lambda'}(\overline{\bq})\prod_{(i,j)\in\lambda_k}\left((-1)^{k}q^{k(i-j)-j}\prod_{l\leq k}q_l^{-l} \right)
\end{equation}
Combining terms in \eqref{eqn:edgefactor}, \eqref{eqn:vertex1}, and \eqref{eqn:vertex2}, \eqref{eqn:glueres} becomes
\begin{equation}\label{eqn:glueorb}
\tilde P^{\cX}_{\rho_0^+,\rho_0^-,(\lambda_0,\dots,\lambda_{n-1})}(\bq)\tilde{E}_\lambda \tilde P^{\cX}_{\rho_\infty^-,\rho_\infty^+,(\lambda_{n-1}',\dots,\lambda_0')}(\overline{\bq})
\end{equation}
where
\[
\tilde{E}_\lambda=\prod_{(i,j)\in\lambda_k}v\left(\prod_{l=1}^kq_l^{(m+1)l}\prod_{l=k+1}^{n-1}q_l^{(m+1)(n-l)}\right)(-1)^{n(m+1)+1}q^{n(m+1)(i-j)+1}
\]
From the observations made in Section \ref{sec:quotients}, the following identity is not hard to prove.
\[
\prod_{(i,j)\in\lambda_k}\left(\prod_{l=1}^kq_l^{(m+1)l}\prod_{l=k+1}^{n-1}q_l^{(m+1)(n-l)}\right)q^{n(m+1)(i-j)}=\prod_{(i,j)\in\bar\lambda}q_{j-i}^{(m+1)(i-j)}
\]
Therefore, \eqref{eqn:glueorb} becomes
\[
 P^{\cX}_{\rho_0^+,\rho_0^-,(\lambda_0,\dots,\lambda_{n-1})}(\bq) E_\lambda^{\cZ,e}  P^{\cX}_{\rho_\infty^-,\rho_\infty^+,(\lambda_{n-1}',\dots,\lambda_0')}(\overline{\bq})
\]
where
\[
E_\lambda^{\cZ,e}=v^{|\lambda|}(-1)^{m|\bar\lambda|}\prod_{(i,j)\in\bar\lambda}q_{j-i}^{-mj+(m+2)i+1}
\]
is the edge term from the gluing algorithm in Theorem \ref{thm:gluing}. This concludes the proof of Theorem \ref{thm:globalcrc}.
\end{proof}

\bibliographystyle{alpha}
\bibliography{biblio}

\end{document}